\documentclass{article}
\usepackage[utf8]{inputenc}
\usepackage{amssymb}
\usepackage{amsmath}
\usepackage{amsthm}

\newtheorem{theorem}{Theorem}[section]
\newtheorem{proposition}[theorem]{Proposition}
\newtheorem{lemma}[theorem]{Lemma}
\newtheorem{fact}[theorem]{Fact}
\newtheorem{corollary}[theorem]{Corollary}

\theoremstyle{definition}
\newtheorem{definition}[theorem]{Definition}
\newtheorem{example}[theorem]{Example}
\newtheorem{notation}[theorem]{Notation}

\usepackage{color}

\begin{document}

\author{Alice Medvedev\thanks{This material is based upon work supported by the National Science Foundation under grant no. DMS-1500976.} \and Alex Van Abel\footnotemark[1]}
\title{Variations on the Feferman-Vaught Theorem, with applications to $\prod_p \mathbb{F}_p$}
\maketitle

\begin{abstract}
Using the Feferman-Vaught Theorem, we prove that a definable subset of a product structure must be a Boolean combination of open sets, in the product topology induced by giving each factor structure the discrete topology. We prove a converse of the Feferman-Vaught theorem for families of structures with certain properties, including families of integral domains. We use these results to obtain characterizations of the definable subsets of $\prod_p \mathbb{F}_p$ -- in particular, every formula is equivalent to a Boolean combination of $\exists \forall \exists$ formulae.
\end{abstract}

\section{Introduction}

The central object of our inquiry is the infinite direct product ring $\prod_p \mathbb{F}_p$, where for each prime $p$, $$\mathbb{F}_p := \mathbb{Z} / p \mathbb{Z}$$ is the finite field with $p$ elements. This ring is associated with the ring of adeles, an important object in number theory. The model theory of the adeles, first considered by Weispfenning in \cite{weis}, is the subject of ongoing work by Derakhshan and Macintyre (\cite{dermacad}, \cite{dermacfef}). Like their work, our paper relies on the celebrated theorem of Feferman and Vaught \cite{fefvau} about definable sets in generalized products of structures. Part of this paper (sections 2 and 3) is a collection of some fairly straightforward and almost certainly known refinements and consequences of the Feferman-Vaught Theorem that we could not find in literature.

Although the model theory of the adeles inspired this paper, no adeles or valued fields appear in it beyond this paragraph. The ring of adeles of $\mathbb{Q}$ is the restricted product of the reals and the $p$-adic completions of the rationals; namely $$\mathbb{A}_\mathbb{Q} := \{a \in \mathbb{R} \times \prod_p \mathbb{Q}_p : a(p) \in \mathbb{Z}_p \mbox{ for all but finitely many } p\}.$$ The product of the $p$-adic valuations on $\mathbb{A}_\mathbb{Q}$ is not a valuation per se, as it takes values in the product of countably many copies of $\mathbb{Z}$ (in particular, the ordering on the value group is not linear). The analog of the residue field is the product $\prod_p \mathbb{F}_p$.


Whereas Derakhshan and Macintyre work with restricted products of valued fields, together with the associated value group and the restricted residue ring, our paper works out the details for the simpler case of unrestricted products of rings. The residue product-of-fields $\prod_p \mathbb{F}_p$ from Derakhshan and Macintyre is our motivating example. In this paper we prove that the definable subsets of $\prod_p \mathbb{F}_p$ are Boolean combinations of $\exists \forall \exists$ sets, and are exactly the sets (atomically) definable in the full generalized product (Definition \ref{genprod}). Derakhshan and Macintyre prove similar results for the ring of adeles in \cite{dermacad}, yet our results do not immediately follow from their work, as they rely heavily on the restrictedness of the adeles. Indeed, our result that the full generalized product is definable in $\prod_p \mathbb{F}_p$ applies more generally to any product of integral domains, and our $\exists \forall \exists$ reduction applies to any product of distinct finite fields. Such products arise, as $\prod_p \mathbb{F}_p$ does, as residue products-of-fields for rings of adeles $\mathbb{A}_K$ for number fields $K$.

There is another reason that the ring $\prod_p \mathbb{F}_p$ is interesting, beyond the connection to the adeles and the excuse to meditate on the Feferman-Vaught Theorem. This structure, as well as the associated atomic Boolean algebra $\mathcal{P}(\Lambda)$ of an infinite set $\Lambda$, are peculiarly tame for something that lies on the ``bad'' side of every stability-theoretic dividing line: they have SOP, the strong order property; IP, the independence property; and TP$_2$, the tree property of the second kind. Yet their definable subsets are still quite well-behaved. This inspires the question: is there some yet-to-be-discovered dividing line which would classify $\prod_p \mathbb{F}_p$ and $\mathcal{P}(\Lambda)$ as ``nice''?

\subsection{Outline of the Paper}

Section 1 contains our motivations, this outline, and some background and notation.

In Section 2, we collect and spell out some folklore around the Feferman-Vaught Theorem \cite{fefvau}. The results are very natural and almost certainly known, the proofs are straightforward, but we could not find them in literature. We first recall the Feferman-Vaught Theorem \cite{fefvau}, in both its original form and in its more common formulation for Cartesian product structures, if only to set up the notation we use throughout the paper. We then show that definable subsets of products are combinatorially well-behaved, and are Boolean combinations of open sets in the product topology (see Definition \ref{dptdef}).

In Section 3, we show that the limitations placed by the Feferman-Vaught Theorem on the definable subsets of $\prod_p \mathbb{F}_p$ are sharp (Theorem \ref{gendef}), i.e. every relation in the full generalized product is definable in the direct product. This is in contrast to the Feferman-Vaught Theorem for products, which could be phrased as saying every definable subset of the direct product is an atomic relation in the full generalized product. The only properties we need of $\prod_p \mathbb{F}_p$ for this result is that each $\mathbb{F}_p$ is an integral domain; in Theorem \ref{gendef} we abstract away the algebra and give sufficient model-theoretic conditions, satisfied by integral domains. Again, this is not deep, but quite useful and not written down anywhere, as far as we know.

In Section 4, we bring together the results in Sections 2 and 3 and apply them specifically to the product structure $\prod_p \mathbb{F}_p$; and, in fact, to the general case of $\prod_{q \in X} \mathbb{F}_q$ where $X$ is any subset of the prime powers. We obtain a quantifier reduction result: every formula is equivalent in $\prod_{q \in X} \mathbb{F}_q$ to a Boolean combination of $\exists \forall \exists$ formulae. This result is much more specific to our structure than the results in previous sections, and uses the results from the early 70's in the model theory of finite fields, as described in the survey paper \cite{chatz}.

\subsection{Notation and Background}

In this paper, we work exclusively with first-order languages, and ``definable'' means ``definable with parameters''. Tuples of elements will be denoted as $\bar{a}$ and sometimes as $(a_1, \ldots, a_n)$; the arity $n$ is usually left implied. As an abuse of notation, we often write $\bar{a} \in \mathcal{A}$ when we formally mean $\bar{a} \in \mathcal{A}^n$.

For an $L$-formula $\varphi(\bar{x})$ and an $L$-structure $\mathcal{A}$, we let $\varphi^\mathcal{A}$ be the set $$\varphi^\mathcal{A} := \{\bar{a} \in \mathcal{A} : \mathcal{A} \models \varphi(\bar{a})\}.$$

Our background reference for model theory is \cite{hodges}, and our background reference for algebra is \cite{dumfoo}.

\subsubsection{Direct Products}
We assume the reader is familiar with the definition of the Cartesian product $\prod_{\lambda \in \Lambda} X_\lambda$ of a family of sets $\{X_\lambda : \lambda \in \Lambda\}$.

We will be concerned with finite tuples of elements from $\prod_{\lambda \in \Lambda} X_\lambda$, and the corresponding finite tuples in each $X_\lambda$, so we introduce the following notation.

\begin{notation}
\label{lambdabrack}
Let $\bar{x} = (x_1, \ldots, x_n) \in (\prod_{\lambda \in \Lambda} X_\lambda)^n$. Then for each $\lambda$, we define $\bar{x}[\lambda] = (x_1(\lambda), \ldots, x_n(\lambda)) \in (X_\lambda)^n$. Futhermore, if $\Lambda_0 \subseteq \Lambda$, then we let $\bar{x}[\Lambda_0]$ denote the element $\bar{y} \in (\prod_{\lambda \in \Lambda_0} X_\lambda)^n$ such that $\bar{y}[\lambda] = \bar{x}[\lambda]$ for each $\lambda \in \Lambda$.
\end{notation}

\begin{definition}[Direct Product Structure]
\label{dpsdef}
Let $L$ be a language and let\\ $\{\mathcal{A}_\lambda : \lambda \in \Lambda\}$ be an indexed collection of $L$-structures. We put an $L$-structure $\mathcal{A}$ on the set $A = \prod_{\lambda \in \Lambda} A_\lambda$ by decreeing, for every atomic formula $\phi$, that $\mathcal{A} \models \phi(\bar{a})$ if and only if $\mathcal{A}_\lambda \models \phi(\bar{a}[\lambda])$ for every $\lambda$. More concretely,
\begin{itemize}
\item For every constant symbol $c$, define $c^\mathcal{A}$ by $c^\mathcal{A}(\lambda) = c^{\mathcal{A}_\lambda}$.
\item For every function symbol $F$ and tuple $\bar{a} \in A$, define \\$(F^\mathcal{A}(\bar{a}))(\lambda) = F^{\mathcal{A}_\lambda}(\bar{a}[\lambda])$.
\item For every relation symbol $R$ and tuple $\bar{a} \in A$, we let $\mathcal{A} \models R(\bar{a})$ exactly when $\mathcal{A}_\lambda \models R(\bar{a}[\lambda])$ for every $\lambda \in \Lambda$.
\end{itemize}
We call the resulting structure the \emph{direct product} of the family $\{\mathcal{A}_\lambda : \lambda \in \Lambda\}$, denoted by $\prod_{\lambda \in \Lambda} \mathcal{A}_\lambda$.
\end{definition}

\subsubsection{Boolean Algebras}
\begin{definition}
$L_{Bool} := \{0,1,\triangledown, \triangle, ^C\}$ is the language of Boolean algebras, with constants $0$ and $1$, binary operations $\triangledown$ and $\triangle$ and a unary operation $^C$.
\end{definition}

The canonical $L_{Bool}$ structure on the power set $\mathcal{P}(I)$ of a set $I$ is
$$0^{\mathcal{P}(I)} = \emptyset, \hspace{10pt} 1^{\mathcal{P}(I)} = I, \hspace{10pt} X^{(C) \hspace{2pt} \mathcal{P}(I)} = I - X,$$
$$X \hspace{2pt}\triangledown^{\mathcal{P}(I)}\hspace{2pt} Y = X \cup Y, \hspace{10pt} X \hspace{2pt} \triangle^{\mathcal{P}(I)} \hspace{2pt} Y = X \cap Y.$$

The following notions from the theory of Boolean algebras are used in this paper.

Let $\mathcal{B}$ be a Boolean algebra.

The partial ordering $\leq$ on $\mathcal{B}$ is defined by $a \leq b$ whenever $a \triangle b = a$.

An \emph{atom} in $\mathcal{B}$ is a nonzero element $a \in \mathcal{B}$ which is minimal in the partial order; i.e., for any other $b \in \mathcal{B}$, if $b \leq a$ then $b = a$ or $b = 0$.

An \emph{atomic Boolean algebra} is a Boolean algebra in which every nonzero element is the supremum of the set of atoms below it.

For a set $\Lambda$, the partial order induced on the Boolean algebra $\mathcal{P}(\Lambda)$ is ordinary set inclusion. The atoms are the singleton subsets $\{\lambda\}$ for $\lambda \in \Lambda$, and $\mathcal{P}(\Lambda)$ is an atomic Boolean algebra, a fact which will be used in Section 2.

A good introduction to Boolean algebras is the book \cite{givhalm}.

\section{The Feferman-Vaught Theorem}

Most of the results in this section may be found in \cite{fefvau}.

Fix a first-order language $L$ and a family of $L$-structures $\{\mathcal{A}_\lambda : \lambda \in \Lambda\}$. Let $\mathcal{A}$ be the direct product structure $\prod_{\lambda \in \Lambda}$.

For an atomic $L$-formula $\phi$, recall that $\mathcal{A} \models \phi(\bar{a})$ if and only if $\mathcal{A}_\lambda \models \phi(\bar{a}[\lambda])$ at every $\lambda$. In the language $L_{Bool}$ of Boolean algebras, $\mathcal{A} \models \phi(\bar{a})$ if and only if $\mathcal{P}(\Lambda) \models K_\phi(\bar{a}) = 1$, where $K_\phi(\bar{a})$ is the set $\{\lambda \in \Lambda : \mathcal{A}_\lambda \models \phi(\bar{a}[\lambda])\}$. In that sense, the truth of $\phi(\bar{x})$ in $\mathcal{A}$ is determined by looking at truth of a formula (in this case, itself) in each $\mathcal{A}_\lambda$ for each $\lambda \in \Lambda$, and then evaluating an $L_{Bool}$-formula at the resulting subset of $\Lambda$. Loosely, the Feferman-Vaught Theorem for Products (Proposition \ref{fvprod}) says that every $L$-formula may be determined in a similar fashion.

To state this more formally, we need some definitions. We begin with the following syntactic device.

\begin{definition}[Acceptable Sequence]
\label{accdef}
An \emph{acceptable sequence} is a sequence of formulae $\xi = \langle \Phi, \theta_1,\ldots,\theta_m \rangle$, where

i) $\Phi$ is an $L_{Bool}$-formula, with free variables (among) $y_1, \ldots, y_m$.

ii) $\theta_i$ is an $L$-formula for each $1 \leq i \leq m$.

The free variables of an acceptable sequence $\xi$ are the free variables of $\theta_1, \ldots, \theta_m$.
\end{definition}

Next, we restate the notation $K_\theta(\bar{a})$ that was used at the beginning of this section. Formally, $K_\theta$ is a function from $(\prod_{\lambda \in \Lambda} \mathcal{A}_\lambda)^n$ to $\mathcal{P}(\Lambda)$. For a tuple $\bar{a} \in \prod_{\lambda \in \Lambda} \mathcal{A}_\lambda$, the set $K_\theta(\bar{a})$ is the set of indices at which $\theta(\bar{a}[\lambda])$ holds.

\begin{definition}[$K_\theta(\bar{a})$]
\label{kthetadef}
For an $L$-formula $\theta(\bar{x})$, define $$K_\theta(\bar{a}) := \{\lambda \in \Lambda : \mathcal{A}_\lambda \models \theta(\bar{a}[\lambda])\}.$$
\end{definition}

The functions $K_\theta$ provide a means of connecting the language $L$ with the language $L_{Bool}$, and allow us to interpret acceptable sequences $\xi(x_1,\ldots,x_n) = \langle \Phi; \theta_1, \ldots, \theta_m \rangle$ as $n$-ary predicates $Q_\xi(x_1,\ldots,x_n)$ on $\prod_{\lambda \in \Lambda} \mathcal{A}_\lambda$ as follows.

\begin{definition}[$\mathcal{Q}_\xi$]
\label{qxidef}
Let $\xi(x_1,\ldots,x_n) = \langle \Phi; \theta_1, \ldots, \theta_m \rangle$. Then the predicate defined by the acceptable sequence $\xi$ is the set $$\mathcal{Q}_\xi^\mathcal{A} := \{\bar{a} \in \prod_{\lambda \in \Lambda} \mathcal{A}_\lambda : \mathcal{P}(\Lambda) \models \Phi(K_{\theta_1}(\bar{a}[\lambda]), \ldots, K_{\theta_m}(\bar{a})[\lambda]).$$
\end{definition}

\begin{definition}[Full Generalized Product]
\label{genprod}
The universe of $\prod_{\lambda \in \Lambda} \mathcal{A}_\lambda$ together with the relations $\mathcal{Q}_\xi^\mathcal{A}$ give a relational structure called the \emph{full generalized product} of the family $\{\mathcal{A}_\lambda : \lambda \in \Lambda\}$.
\end{definition}

As remarked at the beginning of this section, the direct product is definable inside the full generalized product; for any atomic $L$-formula $\phi(\bar{x})$, the predicate $\phi^\mathcal{A}$ is defined by the acceptable sequence $\langle y = 1 ; \phi(\bar{x} \rangle$. Many other product-type constructions may be defined or interpreted inside the full generalized product, although most require expansions of the language $L_{Bool}$. For example, if we wish to require that elements have some property at all but finitely many indices, we add the predicate $Fin$ to $L_{Bool}$. The Full Feferman-Vaught Theorem \ref{fefvaughtfull} holds in its same form for expansions of $L_{Bool}$ as well.

Feferman and Vaught called this next theorem \cite[Theorem~3.1]{fefvau} their ``basic theorem for generalized products''. It is, from one point of view, a quantifier elimination result on the full generalized product.

\begin{fact}[The Full Feferman-Vaught Theorem \cite{fefvau}]
\label{fefvaughtfull}
Let $L_{gen}$ be the relation language consisting of the relation symbols $Q_\xi$, as in Definition \ref{qxidef}, so that the full generalized product is an $L_{gen}$-structure. Then for every $L_{gen}$ formula $\varphi(\bar{x})$, there is an acceptable sequence $\xi$ such that for any family of $L$-structures $\{\mathcal{A}_\lambda : \lambda \in \Lambda\}$, and for any tuple $\bar{a}$ from the full generalized product $\mathcal{A}^*$,  $$\mathcal{A}^* \models \varphi(\bar{a}) \leftrightarrow Q_\xi(\bar{a}).$$ Such a sequence $\xi$ may be found effectively from $\varphi$.
\end{fact}

This next proposition is the formalized version of the remarks at the beginning of the section. It is this form (or a slightly generalized form for reduced products) that the ``Feferman-Vaught Theorem'' is often stated, for example in \cite{hodges}.

\begin{proposition}[The Feferman-Vaught Theorem for Products]
\label{fvprod}
For every $L$-formula $\varphi(\bar{x})$, there are $L$-formulae $\theta_1(\bar{x}), \ldots, \theta_m(\bar{x})$ and an $L_{Bool}$-formula $\Phi(y_1,\ldots,y_m)$ so that for any family of $L$-structures $\{\mathcal{A}_\lambda : \lambda \in \Lambda\}$ and any\\ $\bar{a} \in \prod_{\lambda \in \Lambda} \mathcal{A}_\lambda$, we have $\prod_{\lambda \in \Lambda} \mathcal{A}_\lambda \models \varphi(\bar{a})$ if and only if in the power set of $\Lambda$,\\ $\mathcal{P}(\Lambda) \models \Phi(K_{\theta_1}(\bar{a}), \ldots, K_{\theta_m}(\bar{a})).$ The formulae $\Phi, \theta_1, \ldots, \theta_m$ may be found effectively from $\varphi$.
\end{proposition}
\begin{proof}
We first pass effectively from $\varphi(\bar{x})$ to an equivalent $L_{Bool}^{gen}$-formula $\varphi^*(\bar{x})$ by replacing each atomic subformula with an equivalent acceptable sequence relation $Q_\xi$, as in the remarks after Definition \ref{genprod}. Then the Full Feferman-Vaught Theorem effectively provides an acceptable sequence $\xi = \langle \Phi; \theta_1, \ldots, \theta_m \rangle$ such that $\varphi^*(\bar{x})$ is equivalent to $Q_\xi(\bar{x})$ in every full generalized product. The conclusion of the proposition follows from the definition of the interpretation of $Q_\xi$ in a full generalized product.
\end{proof}

The Feferman-Vaught Theorem for Products says that every $L$-definable subset of a product structure is defined by an acceptable sequence. It is not necessarily true that every acceptable sequence is $L$-definable; the main theorem in Section 3 is that certain technical conditions on the family $\{\mathcal{A}_\lambda : \lambda \in \Lambda\}$ imply that every acceptable sequence \emph{is} internally definable, and that in particular this is true of $\prod_p \mathbb{F}_p$.

The Feferman-Vaught Theorem can be used to show that subsets of $\prod_{\lambda \in \Lambda} \mathcal{A}_\lambda$ which are in a sense ``strongly mixed'' with their complements must be undefinable.

\begin{theorem}
\label{chinese}
Suppose $X \subseteq \mathcal{A}^n$ has the property that for every finite $\Lambda_0 \subseteq \Lambda$ and all $\bar{c} \in \prod_{\lambda \in \Lambda_0} \mathcal{A}_\lambda$, there exist $\bar{a}, \bar{b} \in \prod_{\lambda \in \Lambda} \mathcal{A}_\lambda$ such that $\bar{a}[\Lambda_0] = \bar{b}[\Lambda_0] = \bar{c}$, and $\bar{a} \in X$, and $\bar{b} \notin X$. Then $X$ is not a definable subset of $\mathcal{A}^n$.
\end{theorem}

In Example \ref{copyz}, which was the inspiration for the theorem, we show that the copy $\{n \cdot (1,1,1\ldots) : n \in \mathbb{Z}\}$ of $\mathbb{Z}$ in $\prod_p \mathbb{F}_p$ has this property, and therefore is not definable in the ring.

Our proof of Theorem \ref{chinese}, which appears after Corollary \ref{ndbdry}, passes through a fairly natural topology on $\prod_{\lambda \in \Lambda} \mathcal{A}_\lambda$. In fact, Theorem \ref{chinese} follows straightforwardly from the fact that definable subsets of $\mathcal{A}$ are Boolean combinations of open subsets in this topology (Theorem \ref{topthm}).

We begin with the following lemma, which appears in \cite[Lemma~2.2]{fefvau} and is easily verified. Recall \ref{kthetadef} for the definition of $K_\theta(\bar{a})$.

\begin{lemma}
\label{kcomm}
Let $\varphi(\bar{x})$ and $\psi(\bar{x})$ be $L$-formulae. Then for any $\bar{a} \in \mathcal{A}$,
\begin{itemize}
\item $K_{\varphi \wedge \psi}(\bar{a}) = K_\varphi(\bar{a}) \cap K_\varphi(\bar{a})$ and $K_{\varphi \vee \psi}(\bar{a}) = K_\varphi(\bar{a}) \cup K_\varphi(\bar{a})$,
\item $K_{\neg \varphi}(\bar{a}) = \Lambda - K_{\varphi(\bar{a})}$, and
\item $K_{\exists t \varphi(\bar{x},t)}(\bar{a}) = \bigcup_{c \in \mathcal{A}} K_\varphi(\bar{a},c)$.
\end{itemize}
\end{lemma}

Our next step, Lemma \ref{fvtight}, is a tightening of the Feferman Vaught Theorem for Products (Proposition \ref{fvprod}). This method has previously appeared in the literature; for example, it is used in \cite{burris} to derive 0-1 laws for finite structures (the authors thank Emil Jeřábek for his MathOverflow post \cite{jermo} in which they found this reference). It is also similar to Corollary 9.6.6 in \cite{hodges}. The main tool for this step is a well-known quantifier elimination result for atomic Boolean algebras (For a proof see, for example, \cite[Theorem~6.20]{poizat}.)

\begin{fact}
\label{qefiaba}
The theory of infinite atomic Boolean algebras admits full quantifier elimination in the language $0,1,\hspace{2pt} ^C,\triangle,\triangledown, A_1, A_2, \ldots, A_n, \ldots$; where for each $k$, the unary predicate $A_k(x)$ expresses the (first-order definable) notion ``x is greater than at least $k$ distinct atoms'.'' \cite{poizat}
\end{fact}

For the Boolean algebra $\mathcal{P}(\Lambda)$, the predicate $A_k(X)$ means ``$X$ has at least $k$ elements.''

In this next lemma, we sharpen the Feferman-Vaught Theorem for Products.

\begin{lemma}
\label{fvtight}
For every $L$-formula $\varphi(x_1,\ldots,x_n)$, there are $L$-formulae\\ $\theta_1(\bar{x}), \ldots, \theta_m(\bar{x})$ and a quantifier-free, equality-free formula $\sigma(y_1,\ldots,y_m)$ in the language $\{A_k : k \in \omega\}$ such that for every infinite index set $\Lambda$, every product structure $\mathcal{A} = \prod_{\lambda \in \Lambda} \mathcal{A}_\lambda$, and every $\bar{a} \in \mathcal{A}^n$,
$$\mathcal{A} \models \varphi(\bar{a}) \iff \mathcal{P}(\Lambda) \models \sigma(K_{\theta_1}(\bar{a}),\ldots,K_{\theta_m}(\bar{a})).$$
\end{lemma}
\begin{proof}
Recall that $K_{\theta}(\bar{a})$ is the set $\{\lambda \in \Lambda : \mathcal{A}_\lambda \models \theta(\bar{a}[\lambda])\}$.

The Feferman-Vaught Theorem effectively provides $L$-formulae $\theta_1', \ldots, \theta_{m'}'$ and an $L_{Bool}$-formula $\sigma^{(1)}(y_1,\ldots,y_{m'})$ such that for all $\bar{a} \in \mathcal{A}^n$, $$\mathcal{A} \models \varphi(\bar{a}) \iff \mathcal{P}(\Lambda) \models \sigma^{(1)}(K_{\theta'_1}(\bar{a}), \ldots, K_{\theta'_{m'}}(\bar{a})).$$

As per Fact \ref{qefiaba}, $\sigma^{(1)}$ is equivalent to a quantifier-free ($L_{Bool} \cup \{A_1, A_2, \ldots\})$-formula $\sigma^{(2)}$ modulo $Th(\mathcal{P}(\Lambda))$, i.e. $$\mathcal{P}(\Lambda) \models \forall y_1, \ldots, y_{m'} \hspace{5pt} \sigma^{(1)}(\bar{y}) \leftrightarrow \sigma^{(2)}(\bar{y}).$$

Using Lemma \ref{kcomm}, we remove the non-logical symbols $0,1,\triangle, \triangledown$, and $^C$ out of $\sigma^{(2)}$ and into the formulae $\theta_i$. More specifically, the constant $1$ may be replaced by $K_{x = x}$, and similarly for 0. The function symbols $^C, \triangle, \triangledown$ may be absorbed into the formulae $\theta_i$ as in the following example. $$K_{\phi}(\bar{a}) \triangle (K_\psi(\bar{a}))^C = K_{\phi \wedge \hspace{2pt} \neg \psi}(\bar{a}).$$ Finally, equality may be eliminated using the equivalence $$v_1 = v_2 \iff \neg A_1 ((v_1 \triangle v_2^C) \triangledown (v_2 \triangle v_1^C)).$$ This gives a procedure for finding $m \in \omega$, $L$-formulae $\theta_1(\bar{x}), \ldots, \theta_m(\bar{x})$, and a quantifier-free, equality-free formula $\sigma(y_1,\ldots,y_m)$ in the language $\{A_1, A_2, \ldots\}$ such that for all $\bar{a} \in \mathcal{A}$, $$\mathcal{P}(\Lambda) \models \sigma(K_{\theta_1}(\bar{a}), \ldots, K_{\theta_m}(\bar{a}))$$ $$\iff \mathcal{P}(\Lambda) \models \sigma^{(2)}(K_{\theta'_1}(\bar{a}), \ldots, K_{\theta'_{m'}}(\bar{a}))$$ $$\iff \mathcal{P}(\Lambda) \models \sigma^{(1)}(K_{\theta'_1}(\bar{a}), \ldots, K_{\theta'_{m'}}(\bar{a}))$$ $$\iff \mathcal{A} \models \varphi(\bar{a}).$$
\end{proof}

\begin{corollary}
\label{boolcomb}
Every definable subset of $\mathcal{A}^n$ is a finite Boolean combination of sets of the form $$E_{\varphi, k}(\bar{b}) := \{\bar{a} \in \mathcal{A} : \mathcal{P}(\Lambda) \models A_k(K_{\varphi}(\bar{a},\bar{b}) \}$$ $$= \{\bar{a} \in \mathcal{A} : \mathcal{A}_\lambda \models \varphi(\bar{a}[\lambda], \bar{b}[\lambda]) \mbox{ for at least } k \mbox{ distinct indices } \lambda\}$$ for $L$-formulae $\varphi(\bar{x})$ and natural numbers $k$.
\end{corollary}

We now introduce a topology on the product structure $\mathcal{A}$ which will aid us in proving Theorem \ref{chinese}.

\begin{definition}[The Discrete-Product Topology]
\label{dptdef}
Let $\{X_\lambda : \lambda \in \Lambda\}$ be a family of sets, and let $n > 0$. The \emph{discrete-product topology} on $(\prod_{\lambda \in \Lambda} X_\lambda)^n$ is the topology created by first giving each $X_\lambda$ the discrete topology (in which every singleton, hence every subset, is open), then taking the (Tychnoff) product topology on $\prod_{\lambda \in \Lambda} X_\lambda$, and then again on $(\prod_{\lambda \in \Lambda} X_\lambda)^n$.

More concretely, the discrete-topology has as a basis open sets of the form $$U_{\Lambda_0, b} = \{\bar{x} \in (\prod_{\lambda \in \Lambda} X_\lambda) : \bar{x}[\lambda] = \bar{b}[\lambda] \mbox{ for all } \lambda \in \Lambda_0\}$$ as $\Lambda_0$ ranges over the finite subsets of $\Lambda$ and $\bar{b}$ ranges over $(\prod_{\lambda \in \Lambda_0} X_\lambda)^n$.
\end{definition}

The following proposition is a straightforward unwrapping of definitions which leads to the pleasant property that every definable subset is a Boolean combination of open subsets.

\begin{proposition}
\label{open}
Let $\varphi(\bar{x}, \bar{y})$ be an $L$-formula, let $k$ be a natural number, and let $\mathcal{A}$ be a product structure as above. Then for every $\bar{b} \in \mathcal{A}$, the set $E_{\varphi, k}(\bar{b})$ is an open subset of $\mathcal{A}^n$.
\end{proposition}
\begin{proof}
For $X \subseteq \mathcal{A}^n$ to be open in the product topology is for $X$ to have the property that for any $\bar{a} \in X$, there is a finite set $\Lambda_0 \subseteq \Lambda$ such that for all $\overline{a'}$, if $\overline{a'}[\lambda] = \bar{a}[\lambda]$ for each $\lambda \in \Lambda_0$, then $\overline{a'} \in X$.

Suppose $\bar{a} \in E_{\varphi, k}(\bar{b})$. Then for some $\Lambda_0 = \{\lambda_1, \ldots, \lambda_k\} \subseteq \Lambda$ with $|\Lambda_0| = k$, we have $\mathcal{A}_\lambda \models \varphi(\bar{a}[\lambda], \bar{b}[\lambda])$ for each $\lambda \in \Lambda_0$. If $\overline{a'} \in \mathcal{A}^n$ is such that $\overline{a'}[\lambda] = \bar{a}[\lambda]$ for $\lambda \in \Lambda_0$, then $\mathcal{A}_\lambda \models \varphi(\overline{a'}[\lambda], \bar{b}[\lambda])$ for each $\lambda \in \Lambda_0$, and therefore $\overline{a'} \in E_{\varphi, k}$.

Consequently, $E_{\varphi, k}(\bar{b})$ is an open subset of $\mathcal{A}^n$.
\end{proof}

\begin{theorem}
\label{topthm}
Every definable subset of $\mathcal{A}^n$ is a finite Boolean combination of open sets in the discrete-product topology.
\end{theorem}
\begin{proof}
Combine Proposition \ref{open} and Corollary \ref{boolcomb}.
\end{proof}

\begin{corollary}
\label{ndbdry}
Every definable subset of $\mathcal{A}^n$ has a nowhere dense boundary.
\end{corollary}
\begin{proof}
In any topology, open sets have nowhere dense boundaries, and this property is preserved under taking complements and finite unions. Therefore, every Boolean combination of open sets has a nowhere dense boundary.
\end{proof}

\begin{proof}[Proof of Theorem \ref{chinese}]
Suppose $X \subseteq \mathcal{A}^n$ has the property that for every finite $\Lambda_0 \subseteq \Lambda$ and all $\bar{c} \in \prod_{\lambda \in \Lambda_0} \mathcal{A}_\lambda$, there exist $\bar{a}, \bar{b} \in \prod_{\lambda \in \Lambda} \mathcal{A}_\lambda$ such that $\bar{a}[\Lambda_0] = \bar{b}[\Lambda_0] = \bar{c}$, and $\bar{a} \in X$, and $\bar{b} \notin X$. This is equivalent to the condition that every basic open subset of $\mathcal{A}^n$ has nonempty intersection with both $\mathcal{A}^n - X$ and with $X$. So the interiors of $X$ and of $\mathcal{A}^n - X$ are empty, and the boundary of $X$ is all of $\mathcal{A}^n$. Since $\mathcal{A}^n$ is not nowhere dense, then by Corollary \ref{ndbdry}, $X$ can not be definable.
\end{proof}

We note that our assumption that the index set $\Lambda$ is infinite is not necessary for Proposition \ref{open}, Theorem \ref{topthm} and Corollary \ref{ndbdry}. These all hold trivially when $\Lambda$ is finite, for the topology on $\mathcal{A}^n$ is then the discrete topology, where every set is open and every boundary is empty. Theorem \ref{chinese} also holds vacuously, since (taking $\Lambda_0 = \Lambda$) no such set $X$ can exist.

\begin{example}
\label{allbutfin}
Let $\{\mathcal{A}_\lambda : \lambda \in \Lambda\}$ be an infinite family of structures, with $|\mathcal{A}_\lambda| \geq 2$ for each $\lambda$. Let $a_0$ be any fixed element of $\mathcal{A} = \prod_{\lambda \in \Lambda} \mathcal{A}_\lambda$. Then the sets $$\{a \in \mathcal{A} : a(\lambda) = a_0(\lambda) \mbox{ for infinitely many } \lambda \in \Lambda\}, \mbox{ and}$$ $$\{a \in \mathcal{A} : a(\lambda) = a_0(\lambda) \mbox{ for all but finitely many } \lambda \in \Lambda\}$$ have the property from Theorem \ref{chinese}, and so neither set is first-order definable in $\mathcal{A}$.

In particular, if each $\mathcal{A}_\lambda$ is an Abelian group, then the usual copy of the direct sum $\bigoplus_{\lambda \in \Lambda} \mathcal{A}_\lambda$ inside the direct product $\prod_{\lambda \in \Lambda} \mathcal{A}_\Lambda$ is not definable.
\end{example}

\begin{example}
\label{copyz}
Let us consider the ring $\prod_{p \text{ prime }} \mathbb{F}_p$, and the subring\\ $Z = \{n \cdot (1,1,1,\ldots) : n \in \mathbb{Z}\}$. We claim that $Z$ has the property from Theorem \ref{chinese}.

Fix distinct primes $p_1, \ldots, p_m$ and elements $c(p_i) \in \mathbb{F}_{p_i}$. By the Chinese Remainder Theorem, there is an integer $n$ such that $n \equiv c(p_i) \mod p_i$ for each $i = 1, \ldots, m$. If $a = n \cdot (1,1,1,\ldots)$, then $a \in Z$ and $a(p_i) = c(p_i)$ for each $i = 1, \ldots, m$. On the other hand, we can easily find a nonzero element $b$ of $\prod_p \mathbb{F}_p$ such that $b(p) = 0$ for infinitely many $p$, and $b(p_i) = c(p_i)$ for $i = 1, \ldots, m$. Such an element $b$ can not be an element of $Z$, since if $b = n \cdot (1,1,1,\ldots)$, then $n \neq 0$ but $n$ is divisible by infinitely many primes.

Therefore, $Z$ has the property from Theorem \ref{chinese}, and so $Z$ is not definable in $\prod_p \mathbb{F}_p$.
\end{example}

\section{Interpreting the Full Generalized Product in the Direct Product}

The Feferman-Vaught Theorem for Products says that every definable subset of a product structure is atomically definable in the full generalized product. In this section we prove a converse for certain classes of families of structures, including products of integral domains, such as $\prod_p \mathbb{F}_p$: every definable subset of the full generalized product is already definable in the direct product.

In order to formulate the theorem, we need two notions from model theory: interpretations and positive primitive formulae.

\begin{definition}
A \emph{positive primitive (p.p.)} formula is a first-order formula which is (logically equivalent to) a formula of the form $\exists \bar{y} \bigwedge \Phi$, where $\Phi$ is a set of atomic formulae.
\end{definition}

The utility of positive primitivity comes from the following easily-verified fact.

\begin{fact}[\cite{hodges}, Lemma 9.4.2]
\label{ppform}
Let $\phi(\bar{x})$ be a p.p. formula. \\Let $\{\mathcal{A}_\lambda : \lambda \in \Lambda\}$ be a family of $L$-stuctures, and let $\bar{a}$ be a tuple from $\prod_{\lambda \in \Lambda} \mathcal{A}_\lambda$. Then $\prod_{\lambda \in \Lambda} \mathcal{A}_\lambda \models \phi(\bar{a})$ if and only if $\mathcal{A}_\lambda \models \phi(\bar{a}[\lambda])$ for each $\lambda$. (We say that $\phi$ is preserved and reflected in products.)
\end{fact}

In fact, Fact \ref{ppform} is the only property needed of positive primitive formulae in this paper -- one could replace every instance of ``positive primitive'' with ``preserved and reflected in products''

We also need the model-theoretic notion of an interpretation of one structure in another. For ease of presentation, we give a specialized definition in which we require the interpreted structures to be Boolean algebras.

\begin{definition}
\label{interp}
 Let $\mathcal{A}$ be an $L$-structure. Let $(\mathcal{B}; 0,1,\triangle,\bigtriangledown,^C)$ be a Boolean algebra. We say that $\mathcal{B}$ is \emph{interpretable in} $\mathcal{A}$ if there are formulae $\Theta(x), \Theta_{=}(x,y),$ $\Theta_0(x), \Theta_1(x), \Theta_\triangle(x,y,z), \Theta_\triangledown(x,y,z),$ and $\Theta_C(x,y)$, and there exists a surjection $F : \Theta^\mathcal{A} \to \mathcal{B}$ such that for all $x,y,z \in \Theta^\mathcal{A}$,

 \begin{itemize}

\item $F(x) = F(y)$ if and only if $\mathcal{A} \models \Theta_=(x,y)$

\item $F(x) = 0^\mathcal{B}$ if and only if $\mathcal{A} \models \Theta_0(x)$, and similarly with 1.

\item $F(x) \triangle^\mathcal{B} F(y) = F(z)$ if and only if $\mathcal{A} \models \Theta_\triangle(x,y,z)$, and similarly with the operation $\bigtriangledown$

\item $F(x)^C = F(y)$ if and only if $\mathcal{A} \models \Theta_C(x,y)$.

\end{itemize}

\end{definition}

\begin{example}
\label{intint}
Let $L = \{0,1,+,\cdot,-\}$ be the language of rings, and let $R$ be a commutative ring. It is well-known that the set of \emph{idempotents} of $R$, the elements $x$ such that $x^2 = x$, is a Boolean algebra, where the Boolean meet operation $\triangle$ is given by ring multiplication, and the Boolean complement of the idempotent $x$ is the idempotent $x^C = 1 - x$. This is an interpretation of a Boolean algebra in $R$, in which the surjection $F$ is the identity map, and the defining formulae are
\begin{itemize}
\item $\Theta_\mathcal{B}(x)$ is $ x^2 = x$,
\item $\Theta_=(x,y)$ is $x  =y$,
\item $\Theta_0(x)$ is $x = 0$,
\item $\Theta_1(x)$ is $x = 1$,
\item $\Theta_\triangle(x,y,z)$ is $ x \cdot y = z$,
\item $\Theta_\triangledown(x,y,z)$ is $(1 - x) \cdot (1 - y) = (1 - z)$,
\item $\Theta_C(x,y)$ is $1 - x = y.$
\end{itemize}
We call the resulting Boolean algebra the \emph{Boolean algebra of idempotents of $R$}. When $R$ is an integral domain, the Boolean algebra of idempotents is the two-element algebra $\{0^R,1^R\}$. When $R = \prod_{\lambda \in \Lambda} R_\lambda$ is a product of integral domains $R_\lambda$, the Boolean algebra of idempotents is isomorphic to the power set $\mathcal{P}(\Lambda)$. The idempotents of $R$ are exactly the elements of $\prod_{\lambda \in \Lambda} \{0^{R_\lambda}, 1^{R_\lambda}\}$, so we can interpret $\mathcal{P}(\Lambda)$ in $R$ by defining $F(X)$ to be the ``characteristic function'' of $X$.
\end{example}

When the product structure $\mathcal{A} = \prod_{\lambda \in \Lambda} \mathcal{A}_\lambda$ interprets the Boolean algebra $\mathcal{P}(\Lambda)$, we may ask whether we can define, internal to the structure, the maps $K_\varphi : \mathcal{A}^n \to \mathcal{P}(\Lambda)$ (recall Definition \ref{kthetadef}). More formally, we have the following definition.
\begin{definition}
\label{represent}
Suppose $\{\mathcal{A}_\lambda : \lambda \in \Lambda\}$ is a family of $L$-structures. Suppose further that the Boolean algebra $\mathcal{P}(\Lambda)$ is interpretable in $\mathcal{A} = \prod_{\lambda \in \Lambda} \mathcal{A}_\lambda$ (via $(F, \Theta_\mathcal{B}, \ldots)$. We say that $K_\phi$ \emph{is represented by } the formula $\Upsilon_\phi(\bar{x},\bar{y})$ if for every $\bar{a} \in \mathcal{A}$ and $\bar{b} \in (\Theta_\mathcal{B})^\mathcal{A}$, we have $\mathcal{A} \models \Upsilon_\phi(\bar{a},\bar{b})$ if and only if $F(\bar{b}) = K_\phi(\bar{a})$. We say $K_\phi(\bar{x})$ is represented in the interpretation when it is represented by some formula.
\end{definition}

We are now ready to state the main theorem of this section. The Feferman-Vaught Theorem for Products gives an effective procedure for transforming $L$-formulae into ``equivalent'' acceptable sequences, giving an upper bound on the complexity of definable subsets of a product structure. Our Theorem \ref{gendef} is a sort of converse: an effective procedure for defining acceptable sequence predicates using $L$-formulae, under certain technical conditions which, in particular, hold for the ring $\prod_p \mathbb{F}_p$. Thus, the $\emptyset$-definable subsets of $\prod_p \mathbb{F}_p$ are exactly the atomic predicates of the full generalized product of the family $\{\mathbb{F}_p : p$ prime$\}$.

\begin{theorem}
\label{gendef}
Suppose $\{\mathcal{A}_\lambda : \lambda \in \Lambda\}$ is a family of $L$-structures. Let $\mathcal{A}$ be the direct product $\prod_{\lambda \in \Lambda} \mathcal{A}_\lambda$. Then $(A) \Rightarrow (B) \iff (C) \Rightarrow (D)$, where $(A), (B), (C)$ and $(D)$ are the following conditions.

\begin{enumerate}
\item[$(A)$] The two-element Boolean algebra $\{0,1\}$ is uniformly interpretable in the structures $\mathcal{A}_\lambda$ via \emph{positive primitive} formulae $\Theta_\mathcal{B}, \Theta_=, \Theta_0, \ldots$ (recall Definition \ref{interp}). Suppose further that $\Theta_0$ has the property that for every atomic formula $\phi$, the formula $\Theta_0(\bar{y}) \vee \phi(\bar{x})$ is uniformly equivalent in the structures $\mathcal{A}_\lambda$ to a positive primitive formula $\phi_0(\bar{x}, \bar{y})$.
\item[$(B)$] $\mathcal{P}(\Lambda)$ is interpretable in $\mathcal{A}$ with $K_\phi$ represented (Definition \ref{represent}) for every atomic $L$-formula $\phi$.
\item[$(C)$] $\mathcal{P}(\Lambda)$ is interpretable in $\mathcal{A}$ with $K_\varphi$ represented for \emph{every} $L$-formula $\varphi$.
\item[$(D)$] The full generalized product of the family $\{\mathcal{A}_\lambda : \lambda \in \Lambda\}$ (Definition \ref{genprod}) is definable in $\mathcal{A}$, and the $L$-definable subsets of $\mathcal{A}$ are exactly the sets defined by acceptable sequences.
\end{enumerate}
\end{theorem}

Before we prove Theorem \ref{gendef}, we continue with Example \ref{intint}, which was the blueprint for the theorem and demonstrates the relevance to the product ring $\prod_p \mathbb{F}_p$.

\begin{example}
\label{intgen}
Let $L = (0,1,+,\cdot,-, \ldots)$ extend the language of rings. Suppose $\{R_\lambda : \lambda \in \Lambda\}$ is a collection of integral domains. We can uniformly interpret the two-element Boolean algebra $\{0,1\}$ in the integral domains using the interpretation from Example \ref{intint}. Note that the interpreting formulae are all atomic, hence positive primitive.

If $\phi(\bar{x})$ is an atomic $L$-formula, then $\phi(\bar{x})$ is equivalent to a formula $F(\bar{x}) = 0$ for some polynomial $F \in \mathbb{Z}[\bar{x}]$. Then, because each $\mathcal{A}_\lambda$ is an integral domain, the formula $y = 0 \vee F(\bar{x}) = 0$ is uniformly equivalent to the positive primitive formula $F(\bar{x}) \cdot y = 0$.

Hence, condition (A) from Theorem \ref{gendef} holds $\{R_\lambda : \lambda \in \Lambda\}$, and so the full generalized product is definable in the direct product ring $\prod_{\lambda \in \Lambda} R_\lambda$.

In particular, the definable subsets of the ring $\prod_p \mathbb{F}_p$ are exactly those defined by acceptable sequences. We return to the details of this construction in section 4, where we use this interpretation to prove a quantifier reduction result for $\prod_p \mathbb{F}_p$.
\end{example}

\begin{proof}[Proof of Theorem \ref{gendef}]
Let $\{\mathcal{A}_\lambda : \lambda \in \Lambda\}$ be a family of $L$-structures, and let $\mathcal{A}$ be the direct product $\prod_{\lambda \in \Lambda} \mathcal{A}_\lambda$.\\

$(A) \Rightarrow (B)$: Suppose that the two-element Boolean algebra $\{0,1\}$ is uniformly interpreted in the structures $\mathcal{A}_\lambda$ by positive primitive $L$-formulae $\Theta_\mathcal{B},\\ \Theta_=, \Theta_0, \Theta_1, \Theta_\triangle, \Theta_\triangledown,$ and $\Theta_C$. We want to interpret $\mathcal{P}(\Lambda)$ in $\mathcal{A}$.

Recall that an interpretation of a Boolean algebra $\mathcal{B}$ in a structure $\mathcal{A}$ is specified by a tuple of formulae $(\Theta_\mathcal{B}, \Theta_=, \ldots)$ and a surjection $F: (\Theta_\mathcal{B})^\mathcal{A} \to \mathcal{B}$. We will let $\Theta_\mathcal{B}, \Theta_=, \ldots$ be the same p.p. formulae as in the uniform interpretation of $\{0,1\}$. Since $\Theta_\mathcal{B}$ is positive primitive, by Fact \ref{ppform} we have $\mathcal{A} \models \Theta_\mathcal{B}(\bar{a})$ if and only if $\bar{a}[\lambda] \in (\Theta_\mathcal{B})^{\mathcal{A}_\lambda}$ for every $\lambda$. Since $F_\lambda$ is defined on each $(\Theta_\mathcal{B})^{\mathcal{A}_\lambda}$, it makes sense for us to define $F : (\Theta_\mathcal{B})^\mathcal{A} \to \mathcal{P}(\Lambda)$ by letting $$F(\bar{a}) := \{\lambda \in \Lambda : F_\lambda(\bar{a}[\lambda]) = 1\} $$ $$= \{\lambda \in \Lambda : \mathcal{A}_\lambda \models \Theta_1(\bar{a}[\lambda])\}.$$ For every $\lambda$, the map $F_\lambda : (\Theta_\mathcal{B})^\mathcal{A} \to \{0,1\}$ is surjective, and it follows that $F$ is surjective as well.

Using positive primitivity of the formulae $\Theta$, it is easy to show that \\ $(F, \Theta_\mathcal{B}, \Theta_=, \Theta_0, \Theta_1, \Theta_\triangle, \Theta_\triangledown, \Theta_C)$ is an interpretation of $\mathcal{P}(\Lambda)$ in $\mathcal{A}$. As an example, let $\bar{a}, \bar{b}, \bar{c} \in (\Theta_\mathcal{B})^\mathcal{A}$. Then $\mathcal{A} \models \Theta_\triangle(\bar{a},\bar{b},\bar{c})$ if and only if for every $\lambda$, $$\mathcal{A}_\lambda \models \Theta_\triangle( \hspace{2pt} \bar{a}[\lambda], \hspace{2pt} \bar{b}[\lambda], \hspace{2pt} \bar{c}[\lambda]\hspace{2pt} ),$$ by positive primitivity of $\Theta_\triangle$. This is true if and only if $$F_\lambda(\hspace{1pt} \bar{a}[\lambda] \hspace{1pt} ) \hspace{4pt} \triangle \hspace{4pt}  F_\lambda(\hspace{1pt} \bar{b}[\lambda] \hspace{1pt} ) = F_\lambda(\hspace{1pt} \bar{c}[\lambda] \hspace{1pt})$$ in the two-element Boolean algebra $\{0,1\}$ for every $\lambda$; and by the definition of $F$, this is true if and only if $F(\bar{a}) \cap F(\bar{b}) = F(\bar{c})$ in $\mathcal{P}(\Lambda)$.

To show that the second part of (B) follows from the second part of (A), fix an atomic $L$-formula $\phi(\bar{x})$. By (A), there is a positive primitive formula $\phi_0(\bar{x}, \bar{y})$ such that $\phi(\bar{x}) \vee \Theta_0(\bar{y})$ is equivalent in each $\mathcal{A}_\lambda$ to $\phi_0(\bar{x}, \bar{y})$. We will use $\phi_0$ to create a formula $\Upsilon_\phi$ which represents $K_\phi$. Informally, $\phi_0(\bar{a},\bar{c})$ says ``$F(\bar{c}) \subseteq K_\phi(\bar{a})$'', and we let $\Upsilon_\phi(\bar{a}, \bar{b})$ hold when $F(\bar{b})$ is the largest such $F(\bar{c})$.

Formally, by positive primitivity, $\mathcal{A} \models \phi_0(\bar{a}, \bar{c})$ if and only if for every $\lambda$, $$\mathcal{A}_\lambda \models \Theta_0(\bar{c}[\lambda]) \mbox{  or  } \mathcal{A}_\lambda \models \phi(\bar{a}[\lambda]).$$ This is true if and only if for every $\lambda$, $$\mathcal{A}_\lambda \models \Theta_1(\bar{c}[\lambda]) \mbox{ implies }\mathcal{A}_\lambda \models \phi(\bar{a}[\lambda]).$$ Therefore for $\bar{c} \in (\Theta_\mathcal{B})^\mathcal{A}$, $$\mathcal{A} \models \phi_0(\bar{a},\bar{c}) \mbox{ if and only if } F(\bar{c}) \subseteq K_\phi(\bar{a}).$$

Let $\Upsilon_\phi(\bar{x},\bar{y})$ be the formula $\phi_0(\bar{x},\bar{y}) \wedge \forall \bar{z} (\phi_0(\bar{x},\bar{z}) \rightarrow \Theta_\triangle(\bar{z},\bar{y},\bar{z}))$. Then for $\bar{a} \in \mathcal{A}$ and $\bar{b} \in (\Theta_\mathcal{B})^\mathcal{A}$, we have $\mathcal{A} \models \Upsilon_\phi(\bar{a},\bar{b})$ if and only if $\bar{b}$ interprets the largest element of $\mathcal{P}(\Lambda)$ which is a subset of $K_\phi(\bar{a})$, which is of course $K_\phi(\bar{a})$ itself. Hence $K_\phi(\bar{x})$ is represented by $\Upsilon_\phi$. Therefore $(A) \Rightarrow (B)$.\\

$(B) \Rightarrow (C)$: Suppose $\mathcal{P}(\Lambda)$ is interpretable in $\mathcal{A}$ with $K_\phi$ represented for every atomic $L$-formula $\phi(\bar{x})$. We show that $K_\varphi$ is represented for every $L$-formula $\varphi(\bar{x})$. Our proof is an induction on the complexity of $\varphi$; the base case, when $\varphi$ is atomic, holds by assumption.

For Boolean connectives $\wedge, \vee$ and $\neg$ the induction step is easy, using Lemma \ref{kcomm}. For example, if $K_\varphi$ is represented by $\Upsilon_\varphi$ and $K_\psi$ is represented by $\Upsilon_\psi$, then $K_{\varphi \wedge \neg \psi}$ is represented by $$\Upsilon_{\varphi \wedge \neg \psi}(\bar{x}, \bar{y}) := \hspace{4pt} \exists \bar{u} \exists \bar{v} \exists \bar{w} \hspace{4pt} \Upsilon_\varphi(\bar{x},\bar{u}) \wedge \Upsilon_\psi(\bar{x},\bar{v}) \wedge \Theta_C(\bar{v},\bar{w}) \wedge \Theta_\triangle(\bar{u},\bar{w},\bar{y}).$$

The case of $\exists$ is slightly tricker; it is similar to the construction, from the proof of $(A) \Rightarrow (B)$, of $\Upsilon_\phi$ from $\phi_0$. By Lemma \ref{kcomm}, we know that $K_{\exists \varphi}(\bar{a})$ is the supremum in $\mathcal{P}(\lambda)$ of the sets $K_\varphi(\bar{a},b)$ as $b$ ranges over $\mathcal{A}$. With this in mind, we define the following formulae. $$\Xi_\varphi(\bar{x}, \bar{y}) := \hspace{3pt} \forall z \hspace{2pt} \exists \bar{w} \hspace{4pt} \Upsilon_\varphi(\bar{x},z,\bar{w}) \hspace{3pt} \wedge \hspace{3pt} \Theta_\triangle(\bar{w},\bar{y},\bar{w}),$$ $$\Upsilon_{\exists \varphi}(\bar{x},\bar{y}) := \Xi_\varphi(\bar{x},\bar{y}) \hspace{3pt} \wedge \hspace{3pt} \forall \bar{v} \hspace{3pt} (\Xi_\varphi(\bar{x},\bar{v}) \rightarrow \Theta_\triangle(\bar{y},\bar{v},\bar{y})).$$ Informally, $\Xi_\varphi(\bar{x},\bar{y})$ says that $y$ interprets an element of $\mathcal{P}(\Lambda)$ which contains $K_\varphi(\bar{x},z)$ for every $z \in \mathcal{A}$, and $\Upsilon_{\exists \varphi}(\bar{x},\bar{y})$ says that $\bar{y}$ interprets the \emph{least} such set, which is the supremum $K_{\exists \varphi}$. Therefore $K_{\exists \varphi}$ is represented by $\Upsilon_{\exists \varphi}$.

In this way, $K_\varphi$ is represented for every $L$-formula $\varphi$. Therefore $(B) \Rightarrow (C)$.

The converse $(C) \Rightarrow (B)$ is trivial.\\

$(C) \Rightarrow (D)$: Suppose $\mathcal{P}(\Lambda)$ is interpretable in $\mathcal{A}$, with $K_\varphi$ represented for every $L$-formula $\varphi$. To show that the full generalized product is definable in $\mathcal{A}$, we want the set $$\mathcal{Q}_\xi^\mathcal{A} = \{\bar{a} \in \mathcal{A} : \mathcal{P}(\Lambda) \models \Phi(K_{\theta_1}(\bar{a}), \ldots, K_{\theta_m}(\bar{a})\}$$ to be definable in $\mathcal{A}$ for every acceptable sequence $\langle \Phi; \theta_1, \ldots, \theta_m \rangle$. That this is indeed true follows straightforwardly from the definability of the interpretation $\mathcal{P}(\Lambda)$ and the definability of the functions $K_\varphi$.

Therefore every definable subset of the full generalized product is definable in the direct product. By the Feferman-Vaught Theorem for Products, every definable subset of the direct product is definable in the full generalized product. Hence the definable subsets of the direct product are exactly the definable subsets of the full generalized product, which by the Full Feferman-Vaught Theorem are exactly the sets defined by acceptable sequences.
\end{proof}

\section{Quantifier Reduction in $\prod_p \mathbb{F}_p$}

Throughout the section, let us fix an arbitrary subset $$X \subseteq \{p^n : p \mbox{ prime, } m \in \omega\}.$$ Theorem \ref{gendef} gives a complete characterization of the definable subsets of $\prod_{q \in X} \mathbb{F}_q$, via Example \ref{intgen}. In this section, we tighten this characterization, and obtain the following quantifier reduction.

Then we obtain the following quantifier reduction

\begin{theorem}
\label{qred}
In $\prod_{q \in X} \mathbb{F}_q$, every $L_{Ring}$ formula is equivalent to a Boolean combination of $\exists \forall \exists$ formulae.
\end{theorem}

The mathematics behind this tightening relies heavily on the model theory of finite fields; the results used in this section may be found in the survey paper \cite{chatz}.

\begin{definition}[Pseudo-finite Field]
\label{pfdef}
A \emph{pseudo-finite field} is a field $F$ in which all of the following three conditions hold.
\begin{enumerate}
\item[i] $F$ is perfect.
\item[ii] $F$ has one unique extension of every finite degree.
\item[iii] Every absolutely irreducible variety over $F$ has an $F$-valued point.
\end{enumerate}
\end{definition}

The following result is due to Ax \cite{ax}.

\begin{fact}[\cite{ax}, Theorem 9]
\label{ax}
The following are equivalent for any field $F$.
\begin{enumerate}
\item[$(A)$] $F$ is pseudo-finite, as in Definition \ref{pfdef}.
\item[$(B)$] $F$ is elementarily equivalent to a nonprincipal ultraproduct of distinct finite fields.
\item[$(C)$] $F$ is an infinite model of the theory of finite fields.
\end{enumerate}
Futhermore, a sentence $\varphi$ holds in every pseudo-finite field if and only if it holds in all but finitely many finite fields.
\end{fact}

The next theorem was observed by Kiefe in \cite{kiefe}.

\begin{definition}
A \emph{Kiefe formula} is a Boolean combination of formulae of the form $\exists t \hspace{3pt} P(\bar{x}, t) = 0$.
\end{definition}

\begin{fact}[\cite{kiefe}]
\label{pffield}
Every $L_{Ring}$ formula $\varphi(\bar{x})$ is equivalent modulo the theory of pseudo-finite fields to a Kiefe formula.
\end{fact}

We may use Facts \ref{ax} and \ref{pffield} to obtain the same result for the theory of all finite fields.

\begin{theorem}
\label{primered}
Every $L_{Ring}$ formula $\varphi(\bar{x})$ is uniformly equivalent in \emph{all} finite fields to a Kiefe formula.
\end{theorem}
\begin{proof}
By Facts \ref{ax} and \ref{pffield}, there is a Kiefe formula $\psi'(\bar{x})$ and a natural number $N$ such that if $q \geq N$, then $\mathbb{F}_q \models (\forall \bar{x} \hspace{3pt} \varphi \leftrightarrow \psi')$. This equivalence may break in the fields of cardinality less than $N$.

Let $q = p^n :< N$. We will use Kiefe formulae (in fact, quantifier-free formulae) to isolate the complete types of the elements of $\varphi^{\mathbb{F}_q}$.

It is a fact that if $E/F$ is a Galois extension of fields, then the complete type over $F$ of any tuple $\bar{a}$ from $E$ is isolated by a finite set of polynomial equations -- one can obtain such polynomials from a finite generating set of the ideal of polynomials over $\mathbb{Z}$ which vanish at $\bar{a}$. For each $\bar{a} \in \varphi^{\mathbb{F}_q}$, let $P_{\varphi,\bar{a}} = (f_1, \ldots, f_k)$ be a tuple of polynomials which isolate the complete type of $\bar{a}$, and let us define $$Pol_{\varphi, q} := \{P_{\varphi,\bar{a}} : \bar{a} \in \varphi^{\mathbb{F}_q}\}.$$ Since $\mathbb{F}_q$ is finite, so is $Pol_{\varphi, q}$. Since $\mathbb{F}_q$ is a Galois extension of $\mathbb{F}_p$, we therefore obtain $\mathbb{F}_q \models \varphi(\bar{a})$ if and only if for some $(f_1, \ldots, f_k) \in Pol_{\varphi, q}$, $$\mathbb{F}_q \models f_1(\bar{a}) = 0 \wedge \cdots \wedge f_k(\bar{a}) = 0.$$ Define $\theta_q(\bar{x})$ to be the formula $$\bigvee_{(f_1,\ldots,f_k) \in Pol_{\varphi, q}} (f_1(\bar{x}) = 0 \wedge \cdots \wedge f_k(\bar{x}) = 0).$$ Then in $\mathbb{F}_q$, the quantifier-free formula $\theta_q(\bar{x})$ is equivalent to $\varphi(\bar{x})$.

Define $\Psi_q$ to be the formula $$p \cdot 1 = 0 \hspace{4pt} \wedge \hspace{4pt} \neg (\exists t \hspace{ 3pt } t^q - t - 1 = 0) \hspace{4pt} \wedge \hspace{4pt} (\exists t \hspace{3pt} P_q(t) = 0)$$ where $p \cdot 1$ is $1 + \ldots + 1$ ($p$ times) and $P_q(t)$ is a polynomial of degree $n$ which is irreducible over $\mathbb{F}_p$ (recall $q = p^n$). Then $\Psi_q$ is a Kiefe formula, and for any finite field $K$, $$K \models \Psi_q \hspace{3pt} \mbox{ if and only if } K \cong \mathbb{F}_q.$$

Now let us define $\theta$ as $$\theta(\bar{x}) := \bigwedge_{q < N} \left( \Psi_q \rightarrow \theta_q \left( \bar{x} \right) \right) \wedge \left( \left( \bigwedge_{q < N} \neg \Psi_q \right) \rightarrow \psi' \left( \bar{x} \right) \right).$$ Then $\theta$ is a Kiefe formula, and $\theta \leftrightarrow \varphi$ holds in every finite field.
\end{proof}

\begin{corollary}
\label{boolcor}
Every definable subset of $(\prod_{q \in X} \mathbb{F}_q)^n$ is a Boolean combination of sets of the form $$E_{\psi, k} = \{\bar{a} \in (\prod_p \mathbb{F}_p)^n : \mathbb{F}_p \models \psi(\bar{a}[p]) \mbox{ for at least } k \mbox{ prime powers } q \in X\}$$ where $k \in \omega$ and $\psi$ is a Kiefe formula.
\end{corollary}
\begin{proof}
Combine Theorem \ref{primered} and Corollary \ref{boolcomb}.
\end{proof}

\begin{proof}[Proof of Theorem \ref{qred}]
By Corollary \ref{boolcor}, it suffices to show that when $\psi$ is a Kiefe formula and $k \in \omega$, the set $E_{\psi, k}$ as defined above is defined in $\prod_{q \in X} \mathbb{F}_q$ by an $\exists \forall \exists$ $L_{ring}$-formula.

Recall from Example \ref{intgen} our interpretation of the power set of the primes in the ring $\prod_q \mathbb{F}_q$. The universe of the interpretation is the set of idempotent elements $\{x \in \prod_{q \in X} \mathbb{F}_q : x^2 = x\} = \{x \in \prod_q \mathbb{F}_q : x(p) = 0$ or $1$ for each $q \in X\}$.

Let us introduce some abbreviations.
\begin{itemize}
\item $Id(x)$ is $x \cdot x= x$,
\item $x \sqsubseteq y$ is $x \cdot y = x$,
\end{itemize}

Let $\phi(\bar{x})$ be an atomic $L_{Ring}$-formula. Then $\phi$ is equivalent modulo the theory of commutative rings to a formula of the form $F(\bar{x}) = 0$ for some polynomial $F \in \mathbb{Z}[\bar{x}]$.

When we run the proof of Theorem \ref{gendef} on the family $\{\mathbb{F}_p : q \in X\}$, we represent $K_\phi$ by the formula $$\Upsilon_\phi(\bar{x},y) := \hspace{6pt} Id(y) \wedge( F(\bar{x}) \cdot y = 0) \wedge \forall z (F(\bar{x}) \cdot z = 0 \rightarrow y \sqsubseteq z).$$ This is a $\forall$-formula, and it works for any product of integral domains. Because each $\mathbb{F}_p$ is in fact a field, we may also represent $K_\phi$ by the formula $$\Upsilon'_\phi(\bar{x},y) := \hspace{6pt} \exists z \exists u \exists v \hspace{3pt} Id(z) \wedge (u \cdot v = 1) \wedge (u \cdot F(\bar{x}) = z) \wedge (y = 1 - z).$$ This formula, which says that $y$ is the interpreted Boolean complement of an idempotent which is associate to $F(\bar{x})$, indeed interprets $K_\phi$, as can be easily verified. Therefore the function $K_\phi$ is $\Delta_0$-definable in $\prod_p \mathbb{F}_p$. Because of this, we may use $K_\phi$ as a function symbol in $L$-formulae without affecting the quantifier complexity of formulae that are not quantifier free.

Now, as in the proof of Theorem \ref{gendef}, the function $K_{\exists z \phi(\bar{x},z)}$ is represented by the formula $$\Upsilon_{\exists z  \phi}(\bar{x},y) := \hspace{4pt} \Xi_\phi(\bar{x}, y) \wedge \forall v \hspace{3pt} (\Xi_\phi(\bar{x}, v) \rightarrow y \sqsubseteq v),$$ where $\Xi_\phi(\bar{x}, y)$ is the formula $$\Xi_\phi(\bar{x}, y) := Id(y) \wedge \forall z \hspace{3pt} K_\phi(\bar{x},z) \sqsubseteq y.$$ $\Xi_\phi$ is a $\forall$ formula, and $\Upsilon_{\exists z  F}$ is a $\forall \exists$ formula.

Recall that a Kiefe formula $\psi$ is a Boolean combination of formulae of the form $\exists z \hspace{3pt} F(\bar{x},z) = 0$. As in Lemma \ref{kcomm}, we may represent the formula $\psi$ by changing the logical connectives in $\Psi$ into the interpretations of the $L_{Bool}$ operations. As an example, the Kiefe formula $$\psi(\bar{x}) := \exists z_1 F_1(\bar{x}, z_1) = 0 \wedge \neg (\exists z_2 F_2(\bar{x}, z_2) = 0)$$ is represented by the formula $$\Upsilon_\psi(\bar{x},y) := \forall t_1 \forall t_2 \hspace{3pt} [\Upsilon_{\exists z F_1}(\bar{x},t_1) \wedge \Upsilon_{\exists z F_2}(\bar{x}, t_2)] \rightarrow y = t_1 \cdot (1 - t_2).$$ In this way, every Kiefe formula is represented by a $\forall \exists$-formula.

Finally, let us define the $L$-formula $$At(x) := Id(x) \hspace{4pt} \wedge \hspace{4pt} \forall y \hspace{4pt} ( Id(y) \Rightarrow [x \sqsubseteq y \vee x \cdot y = 0]).$$ Then $At(x)$ is a $\forall$-formula which defines the atomic elements of the Boolean algebra of idempotents. Now for a Kiefe formula $\psi$ and a natural number $k \in \omega$, the set $$E_{\psi, k} = \{\bar{a} \in (\prod_p \mathbb{F}_p)^n : \mathbb{F}_p \models \psi(\bar{a}[p]) \mbox{ for at least } k \mbox{ primes } p\}$$ is defined by the formula $$\exists z_1 \ldots \exists z_k \exists z  \bigwedge_{i \leq k} At(z_i) \hspace{2pt} \wedge \hspace{2pt} \bigwedge_{i < j \leq k} z_i \neq z_j \hspace{2pt} \wedge \hspace{2pt} \bigwedge_{i \leq k} z_i \sqsubseteq z \hspace{2pt} \wedge \hspace{2pt} \Upsilon_\psi(\bar{x}, z),$$ which has quantifier complexity $\exists \forall \exists$.
\end{proof}

\end{document}